\input ./style/arxiv-general.cfg
\documentclass[aop,MSNbibl,number,seceqn,citesort,dvips]{arximspdf}
\makeatletter
   \@ifpackageloaded{graphicx}{}{\usepackage{graphicx}}
\makeatother

\doi{10.1214/11-AOP722} %kopijuoti is PTS
\volume{41}
\issue{3A}
\pubyear{2013}
\firstpage{1218}
\lastpage{1242}
\docsubty{FLA}

\makeatletter
\newcommand{\eqref}[1]{(\ref{#1})}
\newtheorem{theorem}{Theorem}[section]
\newtheorem{lemma}[theorem]{Lemma}
\newtheorem{proposition}[theorem]{Proposition}

\newproclaim{remark}[theorem]{Remark}

\newcommand{\Z}{\mathbb{Z}}
\renewcommand{\P}{\mathbb{P}_p}
\renewcommand{\S}{\mathcal{S}}
\newcommand{\NN}{\mathcal{N}}
\newcommand{\e}{\varepsilon}
\newcommand{\HH}{\mathcal{H}}
\makeatother

\begin{document}
\begin{frontmatter}

\title{Sharp metastability threshold for an anisotropic bootstrap
percolation model}
\runtitle{Anisotropic bootstrap percolation}

\begin{aug}
\author[A]{\fnms{H.} \snm{Duminil-Copin}\thanksref{T1}\ead[label=e1]{hugo.duminil@unige.ch}}
\and
\author[B]{\fnms{A. C. D.} \snm{Van Enter}\corref{}\ead[label=e2]{aenter@phys.rug.nl}}
\thankstext{T1}{Supported in part by ANR Grant BLAN06-3-134462, the
ERC AG CONFRA, as well as by the Swiss {FNS}.}
\runauthor{H. Duminil-Copin and A. C. D. Van Enter}
\affiliation{Universit\'e de Gen\`eve and Rijksuniversiteit Groningen}
\address[A]{D\'epartement de Math\'ematiques\\
Universit\'e de Gen\`eve\\
Gen\`eve\\
Switzerland\\
\printead{e1}} %adresu isvedimo komanda gale!
\address[B]{Johann Bernoulli institute\\
Rijksuniversiteit Groningen\\
9747 AG Groningen\\
The Netherlands\\
\printead{e2}}
\end{aug}

% HISTORY:
\received{\smonth{11} \syear{2010}}
\revised{\smonth{9} \syear{2011}}

% ABSTRACT
%
\begin{abstract} Bootstrap percolation models have been extensively
studied during the two past decades. In this article, we study the
following ``anisotropic'' bootstrap percolation model: the neighborhood
of a point $(m,n)$ is the set
\[
\{(m+2,n),(m+1,n),(m,n+1),(m-1,n),(m-2,n),(m,n-1)\}.
\]
At time 0, sites are occupied with probability $p$. At each time step,
sites that are occupied remain occupied, while sites that are not
occupied become occupied if and only if three of more sites in their
neighborhood are occupied. We prove that it exhibits a sharp
metastability threshold. This is the first mathematical proof of a
sharp threshold for an anisotropic bootstrap percolation model.
\end{abstract}

% KEYWORDS

\end{frontmatter}

%s1 ###
\section{Introduction}\label{sec1}

%s1.1 ###
\subsection{Statement of the theorem}\label{sec11}
Bootstrap percolation models are interesting models for crack
formation, clustering phenomena, metastability and dynamics of glasses.
They also have been used to describe the phenomenon of jamming (see,
for example,~\cite{Toninelli}), and they are a major ingredient in the
study of so-called kinetically constrained models; see, for example,
\cite{GST}. Other applications are in the theory of sandpiles \cite
{FLP10} and in the theory of neural nets
\cite{TE09,Am10}. Bootstrap percolation was introduced in \cite
{ChalupaLeathReich} and has been an object of study for both physicists
and mathematicians. For some of the earlier results, see, for example,
\cite{ADE90,AizenmanLebowitz,CerfCirillo,GravnerMcdonald,CerfManzo,dGLD09,Enter,Schonmann,Sch92}.

The simplest model is the so-called \textit{simple bootstrap
percolation} on $\mathbb{Z}^2$. At time~0, sites of $\mathbb{Z}^2$
are occupied with probability $p\in(0,1)$ independently of each other.
At each time increment, sites become occupied if at least two of their
nearest neighbors are occupied. The behavior of this model is now
well-understood: the model exhibits a sharp metastability threshold.
Nevertheless, slight modifications of the update rule provide
challenging problems, and the sharp metastability threshold remains
open in general. A few models have been solved, including simple
bootstrap percolation and the modified bootstrap percolation in every
dimension, and so-called balanced dynamics in two dimensions \cite
{Holroydboot,Holroydmodboot,BaloghBollobasMorris,HolroydLiggettRomik,BaloghBollobasDuminil-CopinMorris,BBM2,BBM3,Duminil-CopinHolroyd}. The
case of anisotropic dynamics, in which the neighborhood of a point is
not invariant under a ninety-degree rotation in a lattice plane (even
in two dimensions), has so far eluded mathematicians, and even the
scale at which the metastability threshold occurs is not clear.
(Note added in proof: For a general class of models in 3 dimensions, this scale
has since been established in~\cite{EnFe12}.)

In this article, we provide the first sharp metastability threshold for
an anisotropic model. We consider the following model, first introduced
in~\cite{GravnerGriffeath2}. The neighborhood of a point $(m,n)$ is
the set
\[
\{(m+2,n),(m+1,n),(m,n+1),(m-1,n),(m-2,n),(m,n-1)\}.
\]
At time 0, sites are occupied with probability $p$. At each time step,
sites that are occupied remain occupied, while sites that are not
occupied become occupied if and only if three of more sites in their
neighborhood are occupied. We are interested in the behavior (when the
probability $p$ goes to 0) of the (random) time $T$ at which 0 becomes
occupied. For earlier studies of two-dimensional anisotropic models,
whose results, however, fall short of providing sharp results, we refer
to~\cite{ADE90,Duarte,EnterHulshof,GravnerGriffeath2,GravnerGriffeath,Mountford93,Mountford95,Sch90a}.

\begin{theorem}\label{maintheorem}
Consider the dynamics described above, then
\[
\frac1p\biggl(\log\frac1p\biggr)^2\log T\stackrel
{(P)}{\longrightarrow}\frac1{12}\qquad \mbox{when }p\rightarrow0.
\]
\end{theorem}

This model and the simple bootstrap percolation have very different
behavior, as illustrated in Figure \ref{fig1}.

%f1 ###
\begin{figure}

\includegraphics{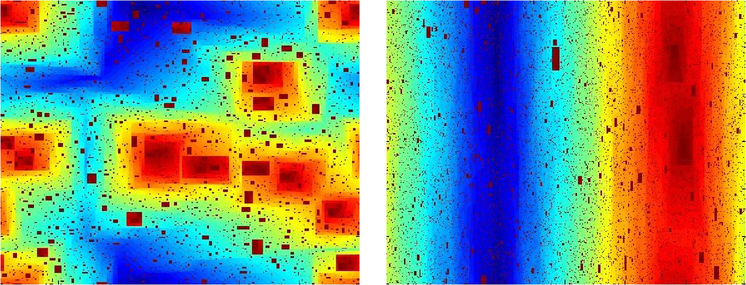}

  \caption{\textit{Left}: An example of a simple bootstrap
percolation's growth (red sites are the oldest, blue the youngest).
\textit{Right}: An example of the anisotropic model's growth.}\label{fig1}
\end{figure}

Combined with techniques of~\cite{Duminil-CopinHolroyd} we believe
that our proof paves the way toward a better understanding of general
bootstrap percolation models. More directly, the following models fall
immediately into the scope of the proof. Consider the neighborhood $\NN
_k$ of $(m,n)$ defined by
\[
\{(m+k,n),\ldots,(m+1,n),(m,n+1),(m-1,n),\ldots,(m-k,n),(m,n-1)\},
\]
and assume that the site $(m,n)$ becomes occupied as soon as $\NN_k$
contains $k+1$ occupied sites. Then the techniques developed in this
article extend to this context, showing that
\[
\frac1p\biggl(\log\frac1p\biggr)^2\log T\stackrel
{(P)}{\longrightarrow}\frac1{4(k+1)} \qquad \mbox{when }p\rightarrow0.
\]

%s1.2 ###
\subsection{Outline of the proof}\label{sec12}

The time at which the origin becomes occupied is determined by the
typical distance at which a ``critical droplet'' occurs; here, a
critical droplet means a connected localized set of occupied sites,
which after growing via the bootstrap rule spans a macroscopic
proportion of the space. Furthermore, the typical distance of this
critical droplet is connected to the probability for a critical droplet
to be created. Such a droplet then keeps growing until it covers the
whole lattice with high probability. In our case, the droplet will be
created at a distance of order $\exp\frac1{6p}(\log\frac1p)^2$.
Determining this distance boils down to estimating how a rectangle,\vspace*{1pt}
consisting of an occupied double vertical column of length $\varepsilon
\frac1p \log\frac1p$, grows to a rectangle of size $1/p^2$ by $\frac
1{3p}\log\frac1p$.

Obtaining an upper bound is usually the easiest part: one must identify
``an almost optimal'' way to create the critical droplet. This way
follows a two-stage procedure. First, a vertical double line of height
$\frac\varepsilon{p}\log\frac1p$ is created.\vspace*{-1pt} Then, the rectangle
grows to size $1/p^2$ by $\frac1{3p}\log\frac1p$. We mention that
this step is quite different from the isotropic case.
Indeed, after starting as a vertical double line, the droplet grows in
a logarithmic manner; that is, it grows logarithmically faster in the
horizontal than in the vertical direction. On the one hand, the
computation of the integral determining the constant of the threshold
is easier than in~\cite{Holroydboot}. On the other hand, the growth
mechanism is more intricate.

The lower bound is much harder: one must prove that our ``optimal'' way
of spanning a rectangle of size $1/p^2$ by $\frac1{3p}\log\frac1p$
is indeed the best one. We combine existing technology with new
arguments. The proof is based on Holroyd's notion of hierarchy applied
to $k$-crossable rectangles containing internally filled sets (i.e.,
sets such that all their sites become eventually occupied when running
the dynamics restricted to the sets). A large rectangle will be
typically created by generations of smaller rectangles. These
generations of smaller rectangles are organized in a tree structure
which forms the hierarchy. In our context, the original notion must be
altered in many different ways (see the proof).

For instance, one key argument in Holroyd's paper is the fact that
hierarchies with many so-called ``seeds'' are unlikely to happen,
implying that hierarchies corresponding to one small seed were the most
likely ones. In our model, this is no longer true. There can be many
seeds, and a new comparison scheme is needed. A second difficulty comes
from the fact that there are stable sets that are not rectangles. We
must use the notion of being $k$-crossed (see Section \ref{sec3}). Even though
it is much easier to be $k$-crossed than to be internally filled, we
can choose the free parameter $k$ to be large enough in order to get
sharp enough estimates. We would like to mention a third difficulty.
Proposition 21 of~\cite{Holroydboot} estimates the probability that a
rectangle $R'$ becomes full knowing that a slightly smaller rectangle
$R$ is full. In this article, we need an analog of this proposition.
However, the proof of Holroyd's Proposition uses the fact that the
so-called ``corner region'' between the two rectangles is unimportant.
In our case, this region matters, and we need to be more careful about
the statement and the proof of the corresponding
proposition.\looseness=-1

The upper bound together with the lower bound result in the sharp
threshold. In a similar way as in ordinary bootstrap percolation
Holroyd's approach refined the analysis of Aizenman and Lebowitz, here
we refine the results of~\cite{GravnerGriffeath2} and~\cite
{EnterHulshof}. We find that the typical growth follows different
``strategies'' depending on which stage of growth we are in. The
logarithmic growth into a critical rectangle is the main new
qualitative insight of the paper. In~\cite{EnterHulshof}, long
vertical double lines were considered as critical droplets; before,
Schonmann~\cite{Schonmann} had identified a single vertical line for
the Duarte model as a possible critical droplet. The fact that these
are not the optimal ones is the main new step toward the identification
of the threshold, apart from the technical ways of proving it. Although
the growth pattern is thus somewhat more complex, the computation of
the threshold can still be performed.\looseness=-1

%s1.3 ###
\subsection{Notations}\label{sec13}

Let $\P$ be the percolation measure with $p>0$. The initial (random)
set of occupied sites will always be denoted by $K$. We will denote by
$\langle K\rangle$ the final configuration spanned by a
set $K$. A set $S$ (for instance a line) is said to be \textit{occupied}
if it contains one occupied site (i.e., $S\cap K\neq\varnothing$). It
is \textit{full} if all its sites are occupied (i.e., $S\subset K$). A
set $S$ is \textit{internally filled} if $S\subset\langle K\cap
S\rangle$. Note that this notation is nonstandard and
corresponds to being internally spanned in the literature.

The neighborhood of 0 will be denoted by $\NN$. Observe that the
neighborhood of $(m,n)$ is $(m,n)+\NN$.

A \textit{rectangle} $[a,b]\times[c,d]$ is the set of sites in $\mathbb
{Z}^2$ included in the Euclidean rectangle $[a,b]\times[c,d]$. Note
that $a$, $b$, $c$ and $d$ do not have to be integers. For a rectangle
$R=[a,b]\times[c,d]$, we will usually\vadjust{\goodbreak} denote by
$(x(R),y(R))=(b-a, d-c)$ the \textit{dimensions} of the rectangle. When
there is no possible confusion, we simply write $(x,y)$. A \textit{line}
of the rectangle $R$ is a set $\{(m,n)\in R:n=n_0\}$ for some $n_0$
fixed. A \textit{column} is a set $\{(m,n)\in R:m=m_0\}$ for some $m_0$ fixed.

%s1.4 ###
\subsection{Probabilistic tools}\label{sec14}

There is a natural notion of increasing events in $\{0,1\}^{\Z^2}$: an
event $A$ is \textit{increasing} if for any pair of configurations
$\omega\leq\omega'$---every occupied site in $\omega$ is occupied
in $\omega'$---such that $\omega$ is in $A$, then $\omega'$ is in
$A$. Two important inequalities related to increasing events will be
used in the proof: the first one is the so-called \textit{FKG
inequality}. Let $A$ and $B$ be two increasing events, then
\[
\P(A\cap B)\geq\P(A)\P(B).
\]
The second is the \textit{BK inequality}. We say that two events occur
disjointly if for any $\omega\in A\cap B$, it is possible to find a
set $F$ so that $\omega_{|F}\in A$ and $\omega_{|F^c}\in B$ (the
restriction means that the occupied sites of $\omega_{|F}$ are exactly
the occupied sites of $\omega$ which are in $F$). We denote the
disjoint occurrence by $A\circ B$ (we denote $A_1\circ\cdots\circ A_n$ for
$n$ events occurring disjointly). Then
\[
\P(A\circ B)\leq\P(A)\P(B).
\]
We refer the reader to the book~\cite{Grimmett} for proofs and a
complete study of percolation models.

We will also use the following easy instance of Chernoff's inequality.
For every $\e>0$, there exists $p_0>0$ such that for every $p<p_0$ and
$n\geq1$, the probability of a binomial variable with parameters $n$
and $p$ being larger than $\e n$ is smaller than~$e^{-n}$.

%s2 ###
\section{\texorpdfstring{Upper bound of Theorem \protect\ref{maintheorem}}{Upper bound of Theorem 1.1}}\label{sec2}

A rectangle $R$ is \textit{horizontally traversable} if in each triplet
of neighboring columns, there exists an occupied site. A~rectangle is
\textit{north traversable} if for any (horizontal) line $\ell=\{
(k,n),k\in\mathbb Z\}$, there exists a site $(m,n)\in R\cap\ell$
such that $\{(m+1,n),(m+2,n),(m,n+1),(m-1,n),(m-2,n)\}$ contains two
occupied sites. It is \textit{south traversable} if for any (horizontal)
line $\ell$, there exists a site $(m,n)\in R\cap\ell$ such that $\{
(m-1,n),(m-2,n),(m,n-1), (m+1,n),(m+2,n)\}$ contains two occupied sites.

\begin{lemma}\label{traversibilityhorizontal}
Let $\varepsilon>0$, then there exist $p_0,y_0>0$ satisfying the
following: for any rectangle $R$ with dimensions $(x,y)$,
\begin{eqnarray*}
\exp[-(1+\varepsilon)xe^{-3py}] &\leq&\P(R\mbox{ is
horizontally traversable})\\
&\leq&\exp[-(1-\varepsilon
)(x-2)e^{-3py}]
\end{eqnarray*}
provided $y_0/p<y<1/(y_0p^2)$ and $p<p_0$.
\end{lemma}

\begin{pf}Let $u=1-(1-p)^y$ be the probability that there exists an
occupied site in a column. Let $A_i$ be the event that the $i$th column\vadjust{\goodbreak}
from the left is occupied. Then $R$ is horizontally traversable if and
only if the sequence $A_1,\ldots,A_x$ has no triple gap (meaning that
there exists $i$ such that $A_i$, $A_{i+1}$ and $A_{i+2}$ do not
occur). This kind of event has been studied extensively; see~\cite
{Holroydboot}. It is elementary to prove that
\[
\alpha(u)^{-x}\leq\P(R\mbox{ is horizontally traversable})\leq
\alpha(u)^{-(x-2)},
\]
where $\alpha(u)$ is the positive root of the polynomial
\[
X^3-uX^2-u(1-u)X-u(1-u)^2.
\]
When $py$ goes to infinity and $p^2y$ goes to 0 (and therefore $p$ goes
to 0), $u$~goes to~1 and
\[
\log\alpha(u)\sim- (1-u)^3\sim-e^{-3py}.
\]
The result follows readily.
\end{pf}

\begin{lemma}\label{traversibilityvertical}
For every $\varepsilon>0$, there exists $p_0,x_0>0$ satisfying the
following property: for any rectangle $R$ with dimensions $(x,y)$,
\[
\exp [(1+\varepsilon)y\log(p^2x)]\leq\P(R\mbox{ is
north traversable})
\]
provided $p<p_0$ and $x\leq1/(x_0p^2)$.
\end{lemma}

The same estimate holds for south traversability by symmetry under reflection.

\begin{pf}Let $n_0\in\mathbb{N}$. Let $v$ be the probability that
one line is occupied. In other words, the probability that there exists
a site $(m,n_0)$ such that two elements of $(m-2,n_0)$, $(m-1,n_0)$,
$(m,n_0+1)$, $(m+1,n_0)$ and $(m+2,n_0)$ are occupied. If $p^2x$ goes
to 0, the probability that there is such a pair of sites is equivalent
to the expected number of such pairs, giving
\[
\log v\sim\log[(8x-16)p^2]\sim\log[p^2x].
\]
(Here $8x-16$ is a bound for the number of such pairs.) Using the FKG
inequality, we obtain
\[
v^y\leq\P(R\mbox{ is north traversable}).
\]
Together with the asymptotics for $v$, the claim follows readily.
\end{pf}

For two rectangles $R_1\subset R_2$, let $I(R_1,R_2)$ be the event that
$R_2$ is internally filled whenever $R_1$ is full. This event depends
only on $R_2\setminus R_1$.

\begin{proposition}\label{lowerboundcrossing}
Let $\varepsilon>0$. Then there exist $p_0,k_0>0$ such that the
following holds: for any rectangles $R_1\subset R_2$ with dimensions
$(x_1,y_1)$ and $(x_2,y_2)$,
\[
\exp\bigl(-(1+\varepsilon)[(x_2-x_1)e^{-3py_2}-(y_2-y_1)\log
(p^2x_1)]\bigr)\leq\P(I(R_1,R_2)),
\]
providing $p<p_0$, $p^2x_1<1/k_0$ and $k_0/p<y_1<1/(k_0p^2)$.\vadjust{\goodbreak}
\end{proposition}

Take two rectangles $R_1\subset R_2$ such that $R_1=[a_1,a_2]\times
[b_1,b_2]$ and $R_2=[c_1,c_2]\times[d_1,d_2]$. Define sets
\begin{eqnarray*}
R_{\ell}&:=&[c_1,a_1]\times[d_1,d_2] \quad\mbox{and} \quad R_{r}:=[a_2,c_2]\times[d_1,d_2],\\
R_{t}&:=&[a_1,a_2]\times[b_2,d_2] \quad\mbox{and}\quad
R_{b}:=[a_1,a_2]\times[d_1,b_1],\\
H&:=& R_2\setminus\{(x,y):x\in[a_1,a_2]\mbox{ or }y\in[b_1,b_2]\}.
\end{eqnarray*}
The set $H$ is the \textit{corner region}; see Figure \ref{figcorner}.

%f2 ###
\begin{figure}

\includegraphics{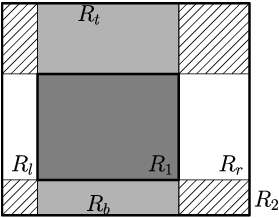}

\caption{The rectangles $R_\ell$ and $R_r$ are in light gray while
$R_t$ and $R_b$ are in white. The corner region~$H$ is hatched.}
\label{figcorner}
\end{figure}

\begin{pf}Let $\e>0$. Set $p_0$ and $k_0:=\max(x_0,y_0)$ in such a
way that Lemmata\ \ref{traversibilityhorizontal} and \ref
{traversibilityvertical} apply. Let $R_1\subset R_2$ two rectangles.
If $R_{\ell}$ and $R_{r}$ are horizontally traversable while $R_{t}$
and $R_{b}$ are respectively north and south traversable, then $R_2$ is
internally filled whenever $R_1$ is internally filled. Using the FKG inequality,
\begin{eqnarray*}
&&\P(I(R_1,R_2))\\
&&\quad\geq\P(R_{\ell}\ \mathrm{hor.\ trav.})\P
(R_{t}\ \mathrm{north\ trav.})\P(R_{r}\ \mathrm{hor.\ trav.})\P(R_{b}\ \mathrm{south\ trav.})\\
&&\quad\geq\exp\bigl(-(1+\varepsilon)
[(x_2-x_1)e^{-3py_2}-(y_2-y_1)\log(p^2x_1)]\bigr),
\end{eqnarray*}
using Lemmata \ref{traversibilityhorizontal} and \ref{traversibilityvertical} (the conditions of these lemmata are satisfied).
\end{pf}

\begin{proposition}[(Lower bound for the creation of a critical
rectangle)]\label{creationofacriticalrectangle}For any $\e>0$,
there exists $p_0>0$ such that for $p<p_0$,
\[
\P([0,p^{-5}]^2\mbox{ is internally filled})\geq\exp
\biggl[-\biggl(\frac16+\e\biggr)\frac1p\biggl(\log\frac1p
\biggr)^2\biggr].
\]
\end{proposition}

\begin{pf}Let $\e>0$. For any $p$ small enough, consider the
sequence of rectangles $(R_n^p)_{k_0\leq n\leq N}$
\[
R_n^p:=\bigl[0,p^{-1-3n/\log(1/p)}\bigr]\times[0,n/p],
\]
where $k_0$ is defined in such a way that Proposition \ref{lowerboundcrossing} applies with $\e$ and $N:=\frac13\log\frac1p-\log k_0$.
The following computation is straightforward, using Proposition~\ref{lowerboundcrossing}:
\begin{eqnarray*}
&&\prod_{n=k_0}^{N}\P[I(R_n^p,R_{n+1}^p)]\\
&&\qquad\geq\exp\Biggl[-(1+\varepsilon)\sum_{n=k_0}^{N}
\biggl(\bigl(p^{-1-3(n+1)/\log(1/p)}-p^{-1-3n/\log(1/p)}\bigr)e^{-3(n+1)}\\
&&\hspace*{185pt}\qquad{}-\frac1p \log\bigl(p^2p^{-1-3n/\log(1/p)}
\bigr)\biggr)\Biggr]\\
&&\qquad=\exp\Biggl(-(1+\varepsilon)\frac1p\biggl(\log\frac1p
\biggr)^2\\
&&\hspace*{32pt}\qquad{}\times\Biggl[\sum_{n=k_0}^{N} \frac{1-e^{-3}}{(\log(1/p)
)^2}+\sum_{n=k_0}^{N} \frac{1}{\log(1/p)}\biggl(1-3\frac{n}{\log(1/p)}\biggr)\Biggr]\Biggr).
\end{eqnarray*}
The first sum goes to 0 as $O(1/ \ln\frac1p)$\vspace*{-2pt} while the second one is
a Riemann sum converging to $\int_0^{1/3} (1-3y)\,dy=\frac16$. The
rectangle $[0,p^{-5}]^2$ is internally filled if all the following
events occur (we include asymptotics when $p$ goes to 0):
\begin{itemize}

\item$E$ the event that $\{0,1\}\times[0,\e\frac1p\log\frac1p]$
is full, of probability $\exp[-2\varepsilon\frac1p(\log
\frac1p)^2]$;

\item$F$ the event that $R=[0,\e\frac1p\log\frac1p]\times
[0,p^{-(1+\e)}]$ is horizontally traversable, of probability larger than

\[
\bigl[1-(1-p)^{\e(1/p)\log(1/p)}\bigr]^{p^{-(1+\e)}}\geq
\exp[-p^{\e-1-\e}]\geq\exp\biggl[-\varepsilon\frac1p\biggl(\log
\frac1p\biggr)^2\biggr];
\]

\item$G$ the intersection of $I(R^p_n,R^p_{n+1})$ for $0\leq n\leq
N-1$, of probability larger than $\exp[-(1+\varepsilon)(\frac
16+\varepsilon)\frac1p(\log\frac1p)^2]$ using the
computation above;

\item$H$ the event that $[0,p^{-2+\varepsilon}]\times[0,6\frac
1p\log\frac1p]$ is north traversable, with probability larger than
\[
\bigl(1-(1-p^2)^{p^{-2+\varepsilon}}\bigr)^{2(1/p)\log(1/p)}\approx
(p^\varepsilon)^{6(1/p)\log(1/p)}=\exp\biggl[-6\e\frac
1p\biggl(\log\frac1p\biggr)^2\biggr];
\]

\item$I$ the event that $[0,6\frac1p\log\frac1p]\times[0,p^{-5}]$\vspace*{-2pt}
is horizontally traversable, with probability larger than
$(1-(1-p)^{6(1/p)\log(1/p)})^{p^{-5}}$ and thus converging to $1$;

\item$J$ the event that $[0,p^{-5}]^2$ is north traversable with
probability larger than $[1-(1-p^2)^{p^{-5}}]^{p^{-5}}$ thus also
converging to 1.
\end{itemize}
The FKG inequality gives
\begin{eqnarray*}
\P([0,p^{-5}]^2\mbox{ is int. filled})&\geq&\P(E\cap
F\cap G\cap H\cap I\cap J)\\
&\geq&\exp\biggl[-(1+\varepsilon)
\biggl(\frac16+10\varepsilon\biggr)\frac1p\biggl(\log\frac1p
\biggr)^2\biggr]
\end{eqnarray*}
when $p$ is small enough.
\end{pf}

\begin{pf*}{Proof of the upper bound in Theorem \protect\ref{maintheorem}}Let $\e
>0$ and consider $A$ to be the event that any line or column of length
$p^{-5}$ intersecting the box $[-L,L]^2$ where $L=\exp[(\frac
1{12}+\e)\frac1p(\log\frac1p)^2 ]$ contains two
adjacent occupied sites. The probability of this event can be bounded
from below.
\[
\P(A)\geq[1-(1-p^2)^{p^{-5}/2}]^{8L^2}\approx\exp
[8L^2e^{-p^{-3}/2}]\longrightarrow1.
\]
(The factor 8 is due to the number of possible segments of length
$p^{-5}$.) Denote by~$B$ the event that there exists a translate of
$[0,p^{-5}]^2$ which is included in $[-L,L]^2$ and internally filled.
Applying Proposition \ref{creationofacriticalrectangle} and
dividing $[-L,L]^2$ into $(Lp^5)^2$ disjoint squares of size $p^{-5}$,
one easily acquires
\begin{eqnarray*}
\P(B)&\geq&1-\bigl[1-e^{-(1/6+\e)(1/p)(\log(1/p))^2}\bigr]^{(Lp^{5})^2}\\
&\approx& 1-\exp\bigl(-(Lp^{5})^2 e^{-(1/6+\e)(1/p) (\log(1/p))^2}\bigr)\longrightarrow1.
\end{eqnarray*}
Moreover, the occurrence of $A$ and $B$ implies that $\log T\leq(\frac
1{12}+2\e)\frac1p(\log\frac1p)^2$ for $p$ small enough. Indeed, a
square of size $p^{-5}$ is filled in less than $p^{-10}$ steps. After
the creation of this square, it only takes a number of steps of order
$\exp[(\frac1{12}+\e)\frac1p(\log\frac1p
)^2]$ to progress and reach 0, thanks to the event $A$. The FKG\vspace*{1pt}
inequality yields
\[
\P\biggl[\log T\leq\biggl(\frac1{12}+2\e\biggr)\frac1p\biggl(\log
\frac1p\biggr)^2\biggr]\geq\P(E\cap F)\geq\P(E)\P(F)\rightarrow1
\]
which concludes the proof of the upper bound.
\end{pf*}

%s3 ###
\section{\texorpdfstring{Lower bound of Theorem \protect\ref{maintheorem}}{Lower bound of Theorem 1.1}}\label{sec3}

%s3.1 ###
\subsection{Crossed rectangles}\label{sec31}

Two occupied points $x,y\in\Z^2$ are \textit{connected} if $x\in y+\NN
$. A set is \textit{connected} if there exists a path of occupied
connected sites with end-points being $x$ and $y$.

Two occupied points $x,y\in\mathbb{Z}^2$ are weakly connected if
there exists $z\in\mathbb{Z}^2$ such that $x,y\in z+\NN$. A set $S$
is \textit{weakly connected} if for any points $x,y\in S$, there exists
a path of occupied weakly connected points with end-points $x$ and $y$.

Let $k>0$. A rectangle $[a,b]\times[c,d]$ is $k$-\textit{vertically
crossed} if for every $j\in[c,d-k]$, the final configuration in
$[a,b]\times[j,j+k]$ knowing that $[a,b]\times[c,d]\setminus
[a,b]\times[j,j+k]$ is full contains a connected path from top to
bottom. A~rectangle $R$ is $k$-\textit{crossed} (or simply crossed) if
it is $k$-vertically crossed and horizontally traversable. Let
$A_k(R_1,R_2)$ be the event that $R_2$ is $k$-vertically crossed
whenever $R_1$ is full. Note that this event is contained in the event
that $R_\ell$ and $R_r$ are traversable, and $R_t$ and $R_b$ are
$k$-vertically crossed.

\begin{lemma}\label{traversibilityverticallower}
For every $\varepsilon>0$, there exist $p_0,Q,k>0$ satisfying the
following property: for any rectangle $R$ with dimensions $(x,y)$,
\begin{eqnarray*}
\P[R\mbox{ is }k\mbox{-vertically crossed}]&\leq&
p^{-k}Q^y\exp[(1-\varepsilon)y\log(p^2x)]\qquad \mbox{when } \frac1p\leq x,\\
\P[R\mbox{ is }k\mbox{-vertically crossed}]&\leq&
p^{-k}\exp[(1-\varepsilon)y\log p]\qquad   \mbox
{when }  x<\frac1p,
\end{eqnarray*}
providing $p<p_0$.
\end{lemma}

\begin{pf}Let $\e>0$ and set $k=\lfloor1/\e\rfloor
$. Consider first the rectangle $[0,x]\times[1,k]$ and the event that
there exists a connected path in the final configuration knowing that
$\mathbb{Z}\times(\Z\setminus[1,k])$ is full.
\end{pf}

\begin{claim*}In the initial configuration, there exist
$A_1,\ldots,A_r$ ($r\leq k-1$) disjoint weakly connected sets such that
$n_1+\cdots+n_r\geq k+r-1$ where $n_i\geq2$ is the cardinality of $A_i$.
\end{claim*}
\begin{pf}
We prove this claim by induction. For $k=2$,
the only way to cross the rectangle is to have a weakly connected set
of cardinality $2$. We define $A_1$ to be this set. For $k\geq3$,
there are three cases:

\textit{Case 1: No sites become occupied after time} 0: It implies
that the crossing from bottom to top is present in the original
configuration. Therefore, there exists a connected set crossing the
strip in the original configuration. Moreover, this set is of
cardinality at least $k$ since it must contain one site in each line at
least. Taking the connected subset of cardinality $k$ to be $A_1$, we
obtain the claim in this case.

\textit{Case 2: The first line or the last line intersects a full
weakly connected set of cardinality} 2: Assume that the first line
intersects a weakly connected set $S$ of cardinality 2. The rectangle
$[2,k]\times[0,x]$ is $(k-1)$-vertically crossed. There exist disjoint
sets $B_1,\ldots,B_r$ satisfying the conditions of the claim. If these
sets are disjoint from $S$, set $A_1=S$, $A_2=B_1,\ldots,A_{r+1}=B_r$.
If one set (say $B_1$) intersects~$S$, we set $A_1=B_1\cup S$,
$A_2=B_2,\ldots,A_r=B_r$. In any case the condition on the cardinality
is satisfied.

\textit{Case 3: Remaining cases}: There must exist three sites in
the same neighborhood (we call this set $S$), spanning $(m,n)\in
[0,x]\times[2,k-1]$ at time 1. The rectangles $[0,x]\times[1,n-1]$
and $[0,x]\times[n+1,k]$ are respectively $(n-1)$-vertically crossed
and $(k-n)$-vertically crossed. If $n\notin\{2,k-2\}$, then one can
use the induction hypothesis in both rectangles,\vadjust{\goodbreak} and perform the same
procedure as before. If $n=2$, then apply the induction hypothesis for
the rectangle above. The same reasoning still applies. Finally, if
$n=k-2$, one can do the same with the rectangle below.
\end{pf}

Let $C=C(k)$ be a universal constant bounding the number of possible
weakly connected sets of cardinality less than $k$ (up to translation).
For any weakly connected set of cardinality $n>1$, we have that the
probability to find such a set in the rectangle $[0,x]\times[1,k]$ is
bounded by $Cp^n(kx)$. We deduce using the BK inequality that
%
%e3.1 ###
\begin{eqnarray}\label{zz}
\P\bigl[[0,x]\times[0,k]\ k\mbox{-vert.
cross.}\bigr]&\leq&\sum_{r=1}^k\Biggl(\sum_{n_1+\cdots+n_r\leq k+r-1}
\prod_{j=1}^r (kC)p^{n_j}x\Biggr)
\nonumber
\\[-8pt]
\\[-8pt]
\nonumber
&\leq&\sum_{r=1}^k (k+r)^r(kC)^r
p^{k-1} (px)^r.
\end{eqnarray}
First assume $px> 1$. We find
\begin{eqnarray*}
\P\bigl[[0,x]\times[0,k]\ k\mbox{-vert. cross.}
\bigr]&\leq&\sum_{r=1}^k (k+r)^r(kC)^r p^{2k-2}x^{k-1}\\
&\leq&(2k^3C)^k(p^2x)^{k-1}
\end{eqnarray*}
since $px>1$. We find
\[
\P\bigl[[0,x]\times[0,k]\ k\mbox{-vert. cross.}\bigr]\leq Q^k\exp
[-(1-\e)k\log(p^2x)]
\]
with $Q=2k^3C$.

Now, we divide the rectangle $R$ into $\lfloor y/k\rfloor
$ rectangles of height $k$. If $R$ is vertically crossed, then all the
rectangles are vertically crossed. Using the previous estimate, we obtain
\begin{eqnarray*}
\P[R\mbox{ is }k\mbox{-vertically crossed}
]&\leq&\prod_{i=1}^{\lfloor y/k\rfloor}\P(R_i\mbox{ is
}k\mbox{-vertically crossed})\\
&\leq& Q^{k\lfloor y/k\rfloor}\exp
\biggl[(1-\varepsilon)k\biggl\lfloor\frac yk\biggr\rfloor\log
(p^2x)\biggr].
\end{eqnarray*}
Using that the rectangle of height $k$ is $k$-vertically crossed with
probability larger than $p^k$, we obtain the result in this case.

If $xp<1$, then we can bound the right-hand term of \eqref{zz} by
$Cp^{k-1}$ and conclude the proof similarly.

Observe that when $x>1/p$, the rectangle will grow in the vertical
direction using $\ell$ disjoint weakly connected pairs of occupied
sites (if it grows by $\ell$ lines). When $x<1/p$, a rectangle will
grow in the vertical direction using one big weakly connected set of
$\ell$ occupied sites. From this point of view, the dynamics is very
different from the simple bootstrap percolation.

\begin{remark}
We have seen in the previous proof that being $k$-vertically crossed
involves only sites included in weakly connected sets of cardinality
two.
\end{remark}

This remark will be fundamental in the following proof.

For two rectangles $R_1\subset R_2$, define
\begin{eqnarray*}W^p(R_1,R_2)&=&\frac{p}{(\log(1/p)
)^2}[(x_2-x_1)e^{-3py_2}]\qquad  \mbox{if } \frac1{p^2}\leq x_2,\\
W^p(R_1,R_2)&=&\frac{p}{(\log(1/p))^2}\\
&&{}\times[(x_2-x_1)e^{-3py_2}-(y_2-y_1)\log(p^2x_2)] \qquad\mbox{if } \frac1p\leq x_2<\frac1{p^2},\\
W^p(R_1,R_2)&=&\frac{p}{(\log(1/p))^2}
[(x_2-x_1)e^{-3py_2}-(y_2-y_1)\log p] \qquad  \mbox
{if }  x_2<\frac1p.
\end{eqnarray*}

\begin{proposition}\label{upperboundcrossing}
Let $\varepsilon,T>0$, there exist $p_0,Q,k>0$
%$p_0,x_0,y_0,k>0$
such that for any $p<p_0$ and any rectangles $R_1\subset R_2$ with
dimensions $(x_1,y_1)$ and $(x_2,y_2)$ satisfying
\[
\frac Tp \log\frac1p\leq y_2\leq\frac1p\log\frac1p,
\]
then
\[
\P[A_k(R_1,R_2)]\leq p^{-2k}Q^{y_2-y_1}\exp
\biggl[-(1-\varepsilon)\frac1p\biggl(\log\frac1p\biggr)^2W
^p(R_1,R_2)\biggr].
\]
\end{proposition}

\begin{pf}
Let $\varepsilon>0$ and set $p_0,Q,y_0,k$ given by Lemmata \ref
{traversibilityverticallower} and \ref{traversibilityhorizontal}
applied with $\varepsilon$. Note that $Q$ can be taken greater than 1.
Consider two rectangles $R_1\subset R_2$ satisfying the conditions of
the proposition. We will use that $p^2y_2$ goes to 0 and $y_2$ goes to
infinity (in particular, $y_0\leq y_2\leq p^{-2}/y_0$). We treat the
case $1/p\leq x_2<1/p^2$; the other cases are similar.

First assume
\[
\varepsilon(y_2-y_1)\log(p^2x_2)\leq-(1-\varepsilon)(x_2-x_1)e^{-3py_2}.
\]
The event $A_k(R_1,R_2)$ is included in the events that $R_t$ and $R_b$
are $k$-vertically crossed (these two events are independent). Using
Lemma \ref{traversibilityverticallower}, we easily deduce the claim via
\begin{eqnarray*}
\P[A_k(R_1,R_2)]&\leq&\exp[(1-\e
)(y_2-y_1)\log(p^2x_2)]\\
&\leq&\exp\bigl[-(1-\varepsilon)^2
\bigl((x_2-x_1)e^{-3py_2}-(y_2-y_1)\log(p^2x_2)\bigr)\bigr].
\end{eqnarray*}

We now assume
\[
\varepsilon(y_2-y_1)\log(p^2x_2)\geq-(1-\varepsilon)(x_2-x_1)e^{-3py_2}.
\]
Let $Y$ be the number of vertical lines containing one occupied site of
$H$ weakly connected to another occupied site. We have
%
%e3.2 ###
\begin{eqnarray}\label{decomposition}\P[A_k(R_1,R_2)]&\leq&
\P[A_k(R_1,R_2)\mbox{ and }
Y\leq\e(x_2-x_1)]
\nonumber
\\[-8pt]
\\[-8pt]
\nonumber
&&{}+ \P
[Y\geq\e(x_2-x_1)].
\end{eqnarray}

\textit{Bound on the second term.} Note that the probability
$\alpha$ that a line contains one occupied site in $H$ weakly
connected with another occupied site behaves like $Cp^2(y_2-y_1)$
(where $C$ is universal) and therefore goes to 0 when $p$ goes to 0.
The probability of $Y\geq\e(x_2-x_1)$ is bounded by the probability
that a binomial variable with parameters $n=x_2-x_1$ and $\alpha
=Cp^2(y_2-y_1)$ is larger than $\frac{\e}{3}(x_2-x_1)$. Invoking
Chernoff's inequality, we find
\[
\P[Y\geq\e(x_2-x_1)]\leq\exp[-(x_2-x_1)]
\]
for $p$ small enough. Since $e^{-3py_2}$ converges to 0, we obtain for
$p$ small enough,
%
%e3.3 ###
\begin{eqnarray}\label{bigY}
&&\P[Y\geq\e(x_2-x_1)]\nonumber\\
&&\qquad\leq\exp\biggl[-\frac
{1-\varepsilon}\varepsilon(x_2-x_1)e^{-3py_2}\biggr]
\\
&&\qquad\leq\exp\bigl[-(1-\varepsilon)
\bigl((x_2-x_1)e^{-3py_2}-(y_2-y_1)\log(p^2x_2)\bigr)\bigr].\nonumber
\end{eqnarray}

\textit{Bound on the first term.} Let $E$ be the event that $R_t$
and $R_b$ are $k$-vertically crossed and $Y\leq\varepsilon(x_2-x_1)$.
We know that
%
%e3.4 ###
\begin{eqnarray}\label{1}
&&\P[A_k(R_1,R_2)\mbox{ and }Y\leq\e
(x_2-x_1)]\nonumber\\
&&\quad=\P[R_{\ell}\mbox{ and }R_r\mbox{ hor. trav.}|E]\P
[E]\hspace*{-35pt}\\
&&\quad\leq\P[R_{\ell}\mbox{ and }R_r\mbox{ hor. trav.}|E
]\P[R_t\mbox{ and }R_b\mbox{ are }k\mbox{-vertically
crossed}].\nonumber\hspace*{-35pt}
\end{eqnarray}

We want to estimate the first term of the last line. Let $\Omega$ be
the (random) set of all pairs of weakly connected occupied sites in
$H$. Conditioning on $E$ corresponds to determining the set $\Omega$
thanks to the remark preceding the proof. Let $\omega$ be a possible
realization of $\Omega$. Slice $R_{\ell}\cup R_{r}$ into $m$
rectangles $R_1,\ldots,R_m$ (with widths $x^{(i)}$) such that:
\begin{itemize}
\item no element of $\omega$ intersects these rectangles;
\item all the lines that do not intersect $\omega$ belong to a
rectangle $R_i$;
\item$m$ is minimal for this property (note that $m\leq2\e[x_2-x_1]$).
\end{itemize}
For each of these rectangles, conditioning on $\{\Omega=\omega\}$
boils down to assuming that there are no full pairs in the corner
region, which is a decreasing event, so that via the FKG inequality,
\[
\P(R_i\mbox{ hor. trav.}|\Omega=\omega)\leq\P(R_i\mbox{ hor.
trav.})\leq\exp\bigl[-(1-\varepsilon)\bigl(x^{(i)}-2\bigr)e^{-3py_2}\bigr].
\]
Since $Y\leq\e(x_2-x_1)$, we know that
\[
x^{(1)}+\cdots+x^{(m)}\geq(1-3\e)(x_2-x_1).
\]
We obtain
\begin{eqnarray*}
&&\P[R_{\ell}\mbox{ and }R_{r}\mbox{ are
hor. trav.}|\Omega=\omega]\\
&&\qquad\leq\P[\mbox{rectangles
}R_i\mbox{ are all hor. trav.}|\Omega=\omega]\\
&&\qquad\leq\prod_{i=1}^m\exp-\bigl[(1-\varepsilon
)\bigl(x^{(i)}-2\bigr)e^{-3py_2}\bigr]\\
&&\qquad\leq\exp[-(1-\e)(1-7\e)(x_2-x_1)e^{-3py_2}].
\end{eqnarray*}
By summing over all possible $\omega$, we find
%
%e3.5 ###
\begin{equation}\qquad\P[R_{\ell}\mbox{ and }R_r\mbox{ hor.
trav.}|E] \leq\exp[-(1-\e)(1-7\e
)(x_2-x_1)e^{-3py_2}].\label{2}
\end{equation}
Using Lemma \ref{traversibilityverticallower} and inequality \eqref
{2}, inequality \eqref{1} becomes
\begin{eqnarray*}
&&\P[A_k(R_1,R_2)\cap\{Y\leq\e(x_2-x_1)\}]\\
&&\qquad\leq\exp[-(1-\e)^2(x_2-x_1)e^{-3py_2}
]p^{-2k}Q^{y_2-y_1}\\
&&\qquad\quad{}\times\exp[-(1-\e)(y_2-y_1)\log(p^2y_2)].
\end{eqnarray*}

The claim follows by plugging the previous inequality and inequality
\eqref{bigY} into inequality \eqref{decomposition}.
\end{pf}

%s3.2 ###
\subsection{Hierarchy of a growth}\label{sec32}

We define the notion of hierarchies, and the specific vocabulary
associated to it. This notion is now well established. We slightly
modify the definition, weakening the conditions imposed in~\cite{Holroydboot}.
\begin{itemize}
\item\textit{Hierarchy, seed, normal vertex and splitter}: A \textit{hierarchy}
$\HH$ is a tree with vertex degrees at most three with
vertices $v$ labeled by nonempty rectangles $R_v$ such that the
rectangle labeled by $v$ contains the rectangles labeled by its
descendants. If the number of descendants of a vertex is 0, it is a
\textit{seed}, and if it is one, it is a $normal$ vertex [we denote by
$u\mapsto v$ if $u$ is a normal vertex of (unique) descendant $v$] and
if it is two or more, it is a \textit{splitter}. Let $N(\HH)$ be the
number of vertices in the tree.

\item\textit{Precision of a hierarchy}: A hierarchy \textit{of
precision} $t$ (with $t\geq1$) is a hierarchy satisfying these
additional conditions:

\begin{longlist}
\item[(1)] if $w$ is a seed, then $y(R_w)<2t$, if $u$ is a normal vertex or a
splitter, $y(R_u)\geq2t$;

\item[(2)] if $u$ is a normal vertex with descendant $v$, then $y(R_u)-y(R_v)
\leq2t$;

\item[(3)] if $u$ is a normal vertex with descendant $v$ and $v$ is a seed or
a normal vertex, then $y(R_u)-y(R_v) >t$;

\item[(4)] if $u$ is a splitter with descendants $v_1,\ldots,v_i$, there exists
$j$ such that $y(R_{u})-y(R_{v_j})> t.$\vadjust{\goodbreak}
\end{longlist}

\item\textit{Occurrence of a hierarchy}: Let $k>0$, a hierarchy
$k$-\textit{occurs} if all of the following events occur \textit{disjointly}:
\begin{longlist}
\item[(1)] $R_w$ is $k$-crossed for each seed $w$;

\item[(2)] $A_k(R_v,R_u)$ occurs for each pair $u$ and $v$ such that $u$ is normal;

\item[(3)] $R_u$ is the smallest rectangle containing $\langle
R_{v_1}\cup\cdots\cup R_{v_j}\rangle$ for every splitter~$u$
($v_1,\ldots,v_i$ are the descendants of $u$).
\end{longlist}
\end{itemize}

\begin{remark} In the literature, the precision of a hierarchy is an
element of~$\mathbb{R}^2$. In our case, we do not need to control the
$x$-coordinate.
\end{remark}

\begin{remark}
We can use the BK inequality to deduce that for any hierarchy $\HH$
and $k\geq1$,
\[
\P[\HH\ k\mbox{-occurs}]\leq\prod_{v \mathrm{seed}}\P
[R_v\ k\mbox{-crossed}]\ \prod_{u\mapsto w}\P
[A_k(R_w,R_u)].
\]
\end{remark}

The following lemmata are classical.

\begin{lemma}[(Number of hierarchies)] \label{numberofhierarchies}Let
$t\geq1$, the number $N_{t}(R)$ of hierarchies of precision $t$ for a
rectangle $R$ is bounded by
\[
N_{t}(R)\leq[x(R)+y(R)]^{c[y(R)/t]}
\]
where $c$ is a prescribed function.
\end{lemma}

\begin{pf}The proof is straightforward once we remark that the depth
of the hierarchy is bounded by $y(R)/t$ (every two steps going down in
the tree, the perimeter reduces by at least $t$). For a very similar
proof, see~\cite{Holroydboot}.
\end{pf}

\begin{lemma}[(Disjoint spanning)] \label{disjointspanning} Let $\S$
be an internally filled and connected set of cardinality greater than
$3$, then there exist $i$ disjoint nonempty \textit{connected} sets
$\S_1,\ldots,\S_i$ with $i\in\{2,3\}$ such that:
\begin{longlist}
\item[(i)] the strict inclusions $\S_1\subset\S,\ldots, \S_i \subset\S$ hold;

\item[(ii)] $\langle\S_1\cup\cdots\cup\S_i\rangle=\S$;

\item[(iii)] $\{\S_1 \mbox{ is internally filled}\} \circ\cdots
\circ\{\S_i \mbox{ is internally filled}\}$ occurs.
\end{longlist}
\end{lemma}

\begin{pf} Let $K$ be finite, $\langle K\rangle$ may
be constructed via the following algorithm: for each time step
$t=0,\ldots,\tau$, we find a collection of $m_t$ connected sets $\S
_1^t,\ldots,\S^t_{m_t}$ and corresponding sets of sites
$K_1^t,\ldots,K^t_{m_t}$ with the following properties:
\begin{longlist}
\item[(i)] $K^t_1,\ldots,K^t_{m_t}$ are pairwise disjoint;

\item[(ii)] $K_i^t \subset K$;\vspace*{1pt}

\item[(iii)] $\S^t_i=\langle K_i^t\rangle$ is connected;

\item[(iv)] if $i\neq j$, then we cannot have $\S_i^t \subset\S^t_j$;\vadjust{\goodbreak}

\item[(v)] $K \subset\S^t \subset\langle K\rangle$ where
\[
\S^t:=\bigcup_{i=1}^{m_t}\S_i^t;
\]

\item[(vi)] $\S^{\tau}=\langle K\rangle$.
\end{longlist}

Initially, the sets are just the individual sites of $K$: let $K$ be
enumerated as $K=\{x_1,\ldots,x_r\}$, and set $m_0=r$ and $\S
_i^0=K_i^0=\{x_i\}$, so that in particular $\S_0=K$.
Suppose that we have already constructed the sets $\S_1^{t},\ldots,\S
_{m_t}^{t}$, then:
\begin{longlist}
\item[(a)] if there exist $j$ sets $\S_{i_1}^t,\ldots, \S^t_{i_j}$ (with $2\leq
j\leq3$) such that the spanned set is connected, set $K'$ to be the
union of the previous sets and $S'$ the spanned set.\vspace*{1pt} We construct the
state $(\S_1^{t+1},K_1^{t+1}),\ldots,(\S
_{m_{t+1}}^{t+1},K_{m_{t+1}}^{t+1})$ at time $t+1$ as follows. From the
list $(\S_1^t,K_1^t),\ldots,(\S_{m_t}^t,K_{m_t}^t)$ at time $t$, delete
every pair $(\S_l^t,K_l^t)$ for which $\S_l^t \subset\S'$. Then add
$(\S',K')$ to the list. Next increase $t$ by 1 and return to step (a).

\item[(b)] else stop the algorithm and set $t=\tau$.
\end{longlist}

Properties (i)--(v) are obviously preserved by this procedure, and
$m_t$ is strictly decreasing with $t$, so the algorithm must eventually stop.

To affirm that property (vi) holds, observe that if $\langle
K\rangle\setminus\S^{\tau}$ is nonempty, then there exists
a site $y\in\langle K\rangle\setminus\S^{\tau}$such
that $y+\NN$ contains $3$ occupied sites in $\S^{\tau}$; otherwise
$y$ would not belong to $\langle K\rangle$ (since $y$
does not belong to $K$). These neighbors must lie in at least two
distinct sets $S_1,\ldots,S_i$ since $y$ is not in $\S^{\tau}$. Observe
that the set spanned by $S_1,\ldots,S_i$ is connected (any spanned site
remains connected to the set that spanned it). Therefore, these sets
are in an $i$-tuplet which corresponds to the case (a) of the algorithm
(since $y$ links the connected components). Therefore the algorithm
should not have stopped at time $\tau$.

Finally, to prove the lemma, note that we must have at least one time
step (i.e., $\tau\geq1$) since the cardinality of $\S$ is greater
than $N$. Considering the last time step of the algorithm (from time
$\tau-1$ to time $\tau$) and sets involved in the creation of $\S
'=\S^{\tau}=\langle K\rangle$. These sets fulfill all
of the required properties.~%
\end{pf}

To any connected set $\S$, one can associate the smallest rectangle,
denoted $[\S]$, containing it. For any $k\geq1$, if $\S$ is
internally filled, then $[\S]$ is $k$-crossed.

\begin{proposition} \label{hierarchy} Let $k\geq1$ and $t\geq3$, and
take any connected set $\S$ which is internally spanned, then some
hierarchy of precision $t$ with root-label $R_r=[\S]$ $k$-occurs.
\end{proposition}

\begin{pf}The proof is an induction on the height (the
$y$-dimension) of the rectangle. Let $\S$ be an internally filled
connected set, and let $R=[\S]$. If $y(R)<2t$, then the hierarchy
with\vadjust{\goodbreak}
only one vertex $r$ and $R_r=R$ $k$-occurs. Consider that $y(R)\geq
2t$, and assume that the proposition holds for any rectangle with
height less than $y(R)$.

First observe that, using Lemma \ref{disjointspanning}, there exist
$m_1$ disjoint connected sets $\S^1_1, \ldots, \S^1_{m_1}$ spanning $\S
$ (with associated rectangles called $R^1_1=[\S^1_1], \ldots,\break
R^1_{m_1}=[\S^1_{m_1}]$). Assume that $R^t_{1},\ldots,R^t_{m_t}$ are
defined; while one of the sets is of cardinality greater than $3$, it
is possible to define iteratively $R^{t+1}_1=[\S
^{t+1}_1],\ldots, R^{t+1}_{m_{t+1}}=[\S^{t+1}_{m_{t+1}}]$. This is
obtained by harnessing Lemma \ref{disjointspanning} iteratively. Stop
at the first time step, called $T$, for which the rectangle $R'$ with
smallest height satisfies $y(R)-y(R') \geq t$ ($R'$ obviously exists
since the height of $R$ is greater than $2t$). Three possibilities can occur:

\textit{Case 1}: $y(R)-y(R') \leq2t$.
Since $R'$ is crossed, the induction hypothesis claims that there
exists a hierarchy of precision $t$, called $\HH'$, with root $r'$ and
root-label $R_{r'}=R'$. Furthermore, the event $A_k(R',R)$ occurs
(since $R$ is $k$-crossed), and it does not depend on the configuration
inside $R'$. We construct $\HH$ by adding the root $r$ with label $R$
to the hierarchy $\HH'$. This hierarchy is indeed a hierarchy of
precision $t$, and it occurs since the event $A_k(R',R)$ is disjoint
from the other events appearing in $\HH'$.

\textit{Case 2}: $y(R)-y(R')>2t$ and $T=1$ (the algorithm stopped at
time~1).
There exist $m_1$ rectangles $R_1,\ldots,R_{m_1}$ corresponding to
connected sets created by the algorithm at time $T=1$. Moreover, $R$ is
the smallest rectangle containing $\langle R_1\cup\cdots\cup
R_{m_1}\rangle$. It is easy to see that events $\{R_i\mbox{ is
}k\mbox{-crossed}\}$ occur disjointly due to the fact that sets $\S
_i$ are disjoint. By the induction hypothesis, there exist $m_1$
hierarchies $\HH_i$ ($i=1,\ldots,m_1$) such that events in these
hierarchies depend only on $\S_i$. The hierarchy created by adding the
root $r$ with label $R_r=R$ is a hierarchy of precision $t$ because
$y(R)-y(R_i)\geq t$ for some $i\leq m_1$ (the algorithm stopped at time~1). Moreover, the hierarchy $k$-occurs since $R$ is the smallest
rectangle containing $\langle R_1\cup\cdots\cup R_{m_1}
\rangle$ and hierarchies $\HH_i$ ($i=1,\ldots,m_1$) $k$-occur disjointly.

\textit{Case 3}: $y(R)-y(R')>2t$ and $T\geq2$.
Consider the rectangle $R''$ from which $R'$ has been created and
denote by $R_1,\ldots,R_m$ its other ``descendants'' (set $R_1=R'$). There
exist hierarchies $\HH_1,\ldots, \HH_m$ associated to $R_1,\ldots,R_m$
which occur disjointly. Consider a root $r$ with label $R$ and a second
vertex $y$ with label $R''$. One can construct a hierarchy through the
process of adding $m+1$ additional edges $(r,y)$ and $(y,r_j)$ for
$j=1,\ldots,m$ where $r_j$ is the root of $\HH_j$. This hierarchy
$k$-occurs. It is therefore sufficient to check that it is a hierarchy
of precision $t$. To do so, notice that $y(R'')-y(R)\leq t$ and
$y(R'')-y(R_1)\geq t$ [since $y(R)-y(R'')\leq t$ and $y(R)-y(R')>2t$].
\end{pf}

%s3.3 ###
\subsection{Proof of the upper bound}\label{sec33}

We\vspace*{1pt} want to bound the probability of a hierarchy $\HH$ with precision
$\frac T{p}\log\frac1p$ to occur. Let
\[
\Lambda^p_T= \inf\Biggl\{\sum_{n=0}^N W^p(R_n,R_{n+1}):(R_n)_{0\leq
n\leq N}\in\mathfrak{D}^p_T\Biggr\},
\]
where $\mathfrak{D}^p_T$ denotes the set of finite increasing
sequences of rectangles $R_0\subset\cdots\subset R_N$ such that:
\begin{itemize}
\item$x_0\leq p^{-1-2T}$;\vspace*{1pt}
\item$y_0\leq\frac{2T}p\log\frac1p$ and $y_N\geq\frac1{3p}\log
\frac1p$;
\item$y_{n+1}-y_n\leq\frac Tp\log\frac1p$ for $n=0,\ldots,N-1$.
\end{itemize}
Before estimating the probability of a hierarchy, we bound $\Lambda
_T^p$ when $p$ and $T$ go to 0.
\begin{proposition}
We have
\[
\liminf_{T\rightarrow0}\liminf_{p\rightarrow0}\Lambda^p_T\geq
\frac1{6}.
\]
\end{proposition}

\begin{pf}
Let $\varepsilon,p>0$. Choose $0<T<\varepsilon$ so that
\[
\int_T^\infty\max\{1-3y,0\}\,dy=\frac16-\varepsilon.
\]
For two rectangles $R$ and $R'$ with dimensions $(x,y)$ [resp.,
$(x',y')$], define $x=p^{-X}$ and $y=\frac Yp (\log\frac
1p)$ where $X,Y\in\mathbb{R}_+$ [resp., $x'=p^{X'}$ and
$y'=\frac{Y'}p(\log\frac1p)$]. With these notations, we obtain
%
%e3.7 ###
%e3.6 ###
\begin{eqnarray}
W^p(R,R')&=&\frac{p}{(\log(1/p)
)^2}(p^{-X'}-p^{-X})e^{-3Y'\log(1/p)}\nonumber\\
&&{}-(Y'-Y)\frac{\inf\{
\log(p^2p^{-X'}),\log p\}}{\log(1/p)}\\
&=&\frac{(p^{1-X'+3Y'}-p^{1-X+3Y'})}{(\log
(1/p))^2}+(Y'-Y)\inf\{2-X',1\}.\nonumber
\end{eqnarray}
Consider a sequence of rectangles $(R_n^p)$ in $\mathfrak{D}^p_T$, then
\begin{eqnarray*}
\sum_{n=0}^N W^p(R_n^p,R_{n+1}^p)&=&\sum_{n=0}^N\frac{
(p^{1-X_{n+1}^p+3Y_{n+1}^p}-p^{1-X_n^p+3Y_{n+1}^p})}{(\log
(1/p))^2}\\
&&{}+\sum_{n=0}^N(Y_{n+1}^p-Y_n^p)\inf\{2-X_{n+1}^p,1\}.
\end{eqnarray*}
First assume that there exists $n$ such that
$1-X_{n+1}^p+3Y_{n+1}^p<-2\varepsilon$ for some values of $p$ going to
0. Since
\[
-p^{1-X_{m}^p+3Y^p_{m+1}}+p^{1-X_{m}^p+3Y^p_{m}}\quad \mbox{and}\quad
p^{1-X_{m+1}^p+3Y^p_{m+1}}-p^{1-X_{m}^p+3Y^p_{m+1}}
\]
are positive for every $m$, the first sum is larger than
\[
\frac{p^{1-X_{n+1}^p+3Y_{n+1}^p}-p^{1-X_0^p+3Y_1^p}}{(\log
(1/p))^2}\geq\frac{p^{-2\varepsilon}-p^{3Y_1^p-\varepsilon
}}{(\log(1/p))^2}\rightarrow\infty.
\]
We can thus assume that for any sequence $\liminf\sum
_{n=0}^{N-1}(1-X_{n+1}^p+3Y_{n+1}^p)\geq-2\varepsilon$. We deduce
\begin{eqnarray*}
&&\liminf\sum_{n=0}^N (Y_{n+1}^p-Y_n^p)\inf\{
2-X_{n+1}^p,1\}\\[-3pt]
&&\qquad\geq\liminf\sum_{n=0}^N (Y_{n+1}^p-Y_n^p)\max\{
1-3Y_{n+1}^p-2\varepsilon, 0\}\\[-3pt]
&&\qquad\geq\int_T^\infty\max\{1-3(y+2T)-2\varepsilon,0\}\,dy=\frac
16-9\varepsilon,
\end{eqnarray*}
where in the last line we used $Y_{n+1}^p\leq Y_n^p+2T$. The claim
follows readily.\vspace*{-3pt}~%
\end{pf}

The two following lemmata are easy (yet technical for the second) but
fundamental in the proof of Proposition \ref{upperboundcrossed}
below. They authorize us to control the probability of a hierarchy even
though there could be many seeds (and even large seeds).\vspace*{-3pt}

\begin{lemma}\label{simplelemma}
Let $k\geq1$, $p>0$ and let $R$ be a rectangle with dimensions
$(x,y)$. Then we have for any $a,b>0$,
\[
\P[R\ k\mbox{-crossed}]\leq p^{-k}\ \P
\bigl[A_k([0,a]\times[0,b],[0,a+x]\times[0,b+y])\bigr].\vspace*{-3pt}
\]
\end{lemma}

\begin{pf}
Simply note that if the rectangle $[a,a+x]\times[b,b+y]$ is crossed,
and $\{0\}\times[a+1,a+k]$ is full, then $A_k([0,a]\times
[0,b],[0,a+x,b+y])$ occurs. The proof follows using the FKG inequality.\vspace*{-3pt}
\end{pf}

Let $N(\mathcal H)$ be the number of vertices in the hierarchy.\vspace*{-3pt}

\begin{lemma}\label{lesssimplelemma}
Let $\varepsilon,T>0$, there exist $p_0,Q,k>0$ such that for $p<p_0$,
we have the following:\vspace*{-1pt} Let $\HH$ be a hierarchy of precision $\frac
T{p}\log\frac1p$ with root label $R$ satisfying $\frac1{3p}\log
\frac1p\leq y(R)\leq\frac1{p}\log\frac1p$, then there exists
$N\geq1$ and $R_0\subset\cdots\subset R_N$ rectangles satisfying the
following properties:
\begin{itemize}
\item$R_N$ has dimensions larger than $R$;
\item$R_0$ has dimensions
\[
\biggl(\sum_{u\mathrm{\ seed}}x(R_u),\sum_{u\mathrm{\ seed}}y(R_u)\biggr);
\]
\item$y(R_{n+1})-y(R_n)\leq\frac{2T}p\log\frac1p$ for every $0\leq
n\leq N-1$;
\item we have
\begin{eqnarray*}&&\prod_{v\mapsto w}\P[A_k(R_w,R_v)]\\[-3pt]
&&\qquad\leq
\bigl(p^{-k}Q^{y(R)}\bigr)^{N(\HH)}\prod_{n=0}^N\exp\biggl[-(1-\e
)W^p(R_n,R_{n+1})\frac1p\biggl(\log\frac1p\biggr)^2\biggr].\vadjust{\goodbreak}
\end{eqnarray*}
\end{itemize}
\end{lemma}

\begin{pf}
Let $\varepsilon,T>0$. Fix $p_0,Q,k>0$ so that Proposition \ref{upperboundcrossing} applies with $\varepsilon$ and $T$. Invoking
Proposition \ref{upperboundcrossing} [note that $\frac Tp\log\frac
1p\leq y(R_v)\leq y(R)$ for a normal vertex], we know
\begin{eqnarray*}
\hspace*{-3pt}\prod_{v\mapsto w}\P[A_k(R_w,R_v)]%\\[-10pt]
%&&\qquad
\leq
\prod_{v\mapsto w}Q^{y(R_w)-y(R_v)}\exp\biggl[-(1-\e
)W^p(R_w,R_v)\frac1p\biggl(\log\frac1p\biggr)^2\biggr].
\end{eqnarray*}
It is thus sufficient to find an increasing sequence of rectangles
satisfying the three first conditions such that
\[
\sum_{v\mapsto w}W^p(R_w,R_v)\geq\sum_{n=0}^N
W^p(R_n,R_{n+1})-kN_\mathrm{splitter}(\HH)\log p,
\]
where $N_\mathrm{splitter}$ is the number of splitters in the hierarchy.
This can be done by induction. If the root of the hierarchy is a seed,
the result is obvious since the sums are empty.

If the root $r$ of the hierarchy is a normal vertex, then we consider
the hierarchy with root $v$ being the only descendant of $r$. By
induction there exists a sequence satisfying all the assumptions. By
setting $R_{N+1}$ with dimensions $x(R_N)+x(R_r)-x(R_v)$ and
$y(R_N)+y(R_r)-y(R_v)$, we obtain from the decreasing properties of
$W^p(\cdot,\cdot)$ that $W^p(R_N,R_{N+1})\leq W^p(R_v,R_r)$. The
claim follows readily.

If the root $r$ of the hierarchy is a splitter, then we consider the
hierarchies with roots $v_1,\ldots, v_i$ ($i\in\{2,3\}$) being the
descendant of $r$. There exist sequences $(R^{(i)}_n)_{n\leq N_i}$ for
each of these hierarchies. Consider the following
sequence:
\[
 R_1^{(1)},\ldots,R_{N_1}^{(1)}, R_{N_1}^{(1)}+R_{1}^{(2)},\ldots,R_{N_1}^{(1)}+\cdots
+R_{N_i}^{(i)}, \max\bigl(R,R_{N_1}^{(1)}+\cdots+R_{N_i}^{(i)}\bigr),
\]
where $R+R'$ is any rectangle with dimensions being the sum of the
dimensions of $R$ and $R'$, and $\max(R,R')$ is a rectangle with
dimensions being the maximum of the dimensions of $R$ and $R'$.

Since the dimensions of $R_{N_1}^{(1)}+\cdots+R_{N_i}^{(i)}$ can exceed
those of $R$ by at most~3 (some space is allowed when combining two or
more sets: they are only weakly connected), we obtain via a simple
computation that
\[
W^p\bigl[R_{N_1}^{(1)}+\cdots+R_{N_i}^{(i)},\max\bigl(R,R_{N_1}^{(1)}+\cdots
+R_{N_i}^{(i)}\bigr)\bigr]\leq-k\log p.
\]
We deduce that removing the last rectangle in the sequence costs at
most $k\log p$. The sequence then satisfies all the required conditions.
\end{pf}

\begin{proposition}\label{upperboundcrossed}
Let $\varepsilon>0$. Then there exist $k,p_0>0$ such that
\[
\P(\S\mbox{ is internally filled})\leq\exp\biggl[-\biggl(\frac
16-\e\biggr)\frac1p\biggl(\log\frac1p\biggr)^2\biggr]\vspace*{6pt}
\]
for $p<p_0$ and $\S$ a connected set satisfying $\frac1{3p}\log\frac
1p\leq y([\S])\leq\frac1{p}\log\frac1p$.\vadjust{\goodbreak}
\end{proposition}

\begin{pf}Choose $T,p_0>0$ such that $\Lambda^p_T\geq\frac
16-\varepsilon$ for any $p<p_0$. Consider a connected set $\S$
satisfying the conditions of the proposition, and set $R=[\S]$. First
assume that $x(R)\geq p^{-5}$. Then a simple computation implies that
\begin{eqnarray*}
\P[\S\mbox{ internally filled}]&\leq&\P[R\mbox{ hor.
traversable}]\\
&\leq&\exp\bigl[-(1-\e)p^{-5}e^{-3\log(1/p)}
\bigr]\leq\exp[-p^{-2}],
\end{eqnarray*}
and the claim follows in this case. Therefore, we can assume that
$x(R)\leq p^{-5}$.

Since $\S$ is internally filled, there exists a hierarchy $\HH$ of
precision $\frac T{p}\log\frac1p$ with root label $R$ that occurs
(Proposition \ref{hierarchy}). Moreover, the number of possible
hierarchies is bounded by
\[
N_{t}(R)\leq\biggl[p^{-5}+\frac1p\log\frac1p\biggr]^{c[
(1/p)\log(1/p)/( (T/p)\log(1/p))]}\leq p^{-6c(1/T)}
\]
when $p$ is small enough (Lemma \ref{numberofhierarchies}). We deduce
\begin{eqnarray*}
&&\P(R\mbox{ is crossed})\\
&&\qquad\leq p^{-6c(1/T)}\max\biggl\{\P[\HH\mbox{ occurs}]\
:\ \HH\mbox{ of precision }\frac T{p}\log\frac1p\mbox{ and root
}R\biggr\}.
\end{eqnarray*}

Bounding the probability of $\S$ being internally filled boils down to
estimating the probability for a hierarchy to occur.

\begin{claim*} The probability that a hierarchy $\HH$ of precision
$\frac T{p}\log\frac1p$ with root label~$R$ is $k$-occurring is
smaller than $\exp[-(\frac16-3\e)\frac1p(\log\frac1p)^2]$.
\end{claim*}

\begin{pf}
Let $\HH$ be a hierarchy. First assume that
there exists one seed with root label $R'$ satisfying $x(R')\geq
p^{-1-2T}$. Then the probability that this seed is horizontally
traversable is smaller than $\exp-p^{-1-T}$ when $p$ is small enough
(same computation as usual). The claim follows easily in this case.

We now assume that $x(R')\leq p^{-1-2T}$ for every seed of the
hierarchy. Using Lemma \ref{lesssimplelemma}, there exist rectangles
$R_0\subset\cdots\subset R_N$ satisfying the conditions of the lemma
such that
\begin{eqnarray*}
&&\prod_{v\mapsto w}\P[A_k(R_u,R_v)]\\
&&\qquad\leq
\bigl(p^{-k}Q^{y(R)}\bigr)^{N(\HH)}\prod_{n=0}^N\exp\biggl[-(1-\e
)W^p(R_n,R_{n+1})\frac1p\biggl(\log\frac1p\biggr)^2\biggr].
\end{eqnarray*}
Using Lemma \ref{simplelemma}, we can transform this expression into
\[
\prod_{u\ \mathrm{seed}}\P[R_u\ \mathrm{crossed}]\leq
p^{-kN(\HH)}\prod_{n=1}^{N_\mathrm{seed}(\HH)}\P[A_k(\tilde
{R}_i,\tilde{R}_{i+1})],
\]
where $N_\mathrm{seed}(\HH)\leq N(\HH)$ is the number of seeds of $\HH$ and
\[
\tilde{R}_i=[0,x(R_{u_1})+\cdots+x(R_{u_i})]\times
[0,y(R_{u_1})+\cdots+y(R_{u_i})].\vadjust{\goodbreak}
\]
(In the previous formula, we have indexed the seeds by
$u_1,\ldots,u_{\tilde{N}}$.) We conclude that
\begin{eqnarray*}
&&\P[\HH\mbox{ occurs}]\\
&&\qquad\leq\prod_{u\ \mathrm{seed}}\P[R_u\mbox{ crossed} ]\prod_{u\mapsto v}\P
[A_k(R_u,R_v)]\\
&&\qquad\leq\bigl(p^{-2k}Q^{y(R)}\bigr)^{N(\HH)}\\
&&\quad\qquad{}\times \exp\Biggl\{-(1+\e
)\Biggl[\sum_{n=0}^{\tilde{N}}W^p(\tilde{R}_i,\tilde{R}_{i+1})+\sum
_{n=0}^NW^p(R_i,R_{i+1})\Biggr]\frac1p\biggl(\log\frac1p
\biggr)^2\Biggr\}\\
&&\qquad\leq e^{\e(1/p)(\log(1/p))^2}\exp\biggl[-(1-\e)\biggl(\frac
16-\e\biggr)\frac1p\biggl(\log\frac1p\biggr)^2\biggr]
\end{eqnarray*}
since the sequence $\tilde{R}_0,\ldots,\tilde{R}_{\tilde
{N}},R_0,\ldots,R_N$ is in $\mathfrak{D}_T^p$ for $p$ small enough (we
have excluded the case where the seeds are too large), the number of
vertices is uniformly bounded by $3^{2/T}$ when $p$ goes to 0 and
$y(R)\leq\frac1p\log\frac1p$.
\end{pf}

We can now conclude by the following computation:
\begin{eqnarray*}\P(\S\mbox{ internally filled})
&\leq& p^{-kc(1/T)}\exp\biggl(-\biggl(\frac16-3\e\biggr)\frac
1p\biggl(\log\frac1p\biggr)^2\biggr)\\
&\leq&\exp\biggl[-\biggl(\frac16-4\e\biggr)\frac1p\biggl(\log
\frac1p\biggr)^2\biggr]
\end{eqnarray*}
for $p$ small enough.
\end{pf}

\begin{pf*}{Proof of lower bound in Theorem \protect\ref{maintheorem}}Let $p_0,k>0$
be such that
\[
\P(\S\mbox{ internally filled})\leq e^{-(1/6-\e)(1/p)
(\log(1/p))^2}
\]
for any $p<p_0$ and any connected set $\S$ such that
\[
\frac1{3p}\log\frac1p\leq y([\S])\leq\frac{1}{p}\log\frac1p.
\]

Let $E$ be the event that the origin is spanned by the configuration\vspace*{-1pt} in
$[-\frac1p\log\frac1p, \frac1p\log\frac1p]^2$,\vspace*{1pt} the probability of
this event goes to 0. Indeed, two cases are possible. Either the origin
is occupied, which occurs with probability $p$ going to 0, or the
origin is not occupied at time 0. In this case, there must exist a
neighborhood containing 3 occupied sites at distance less than $\frac
1p\log\frac1p$ to the origin. This probability goes to 0 when $p$
goes to 0.

Let $F$ be the event\vspace*{-1pt} that no rectangle in $[-L,L]^2$ of perimeter
between $\frac1{3p}\log\frac1p$ and $\frac1{p}\log\frac1p$ is
crossed, where
\[
L=\exp\biggl[\biggl(\frac1{12}-\e\biggr)\frac1p\biggl(\log\frac
1p\biggr)^2\biggr].
\]
In this case, Lemma \ref{disjointspanning} implies that no connected
set $\S\subset[-L,L]^2$ such that $y([\S])\geq\frac1p\log\frac
1p$ is crossed.\vspace*{-1pt} Indeed, if it was the case, we could construct a
connected set $\S$ with $\frac1{3p}\log\frac1p\leq y([\S])\leq
\frac{1}{p}\log\frac1p$ which is internally filled using Lemma \ref
{disjointspanning}.

It is easy to see that if $E$ and $F$ hold, we must have $\log T\geq
(\frac1{12}-\e)\frac1p(\log\frac1p)^2$. Indeed, the
information must necessarily come from outside the box $[-L,L]^2$. Then
\begin{eqnarray*}
&&\P\biggl[\log T\leq\biggl(\frac1{12}-\e\biggr)\frac1p\biggl(\log
\frac1p\biggr)^2\biggr]\\
&&\qquad\leq\P(E)+\P(F)\\
&&\qquad\leq\P(E)+p^{-8}e^{2(1/{12}-\e)(1/p)(\log
(1/p))^2}e^{-(1/6-\e)(1/p)(\log(1/p))^2},
\end{eqnarray*}
where $p^{-8}$ bounds the number of possible dimensions of $[\S]$ and
$\exp[2(\frac1{12}-\e)\frac1p(\log\frac1p)^2]$ the
number of possible locations for its bottom-left corner. When $p$ goes
to 0, the right-hand side converges to 0 and the lower bound follows.
\end{pf*}

\begin{remark} Recently, improvements of estimates for the threshold
of simple bootstrap percolation have been proved~\cite{Morris,GravnerHolroydMorris}.
It is an interesting question to try to improve
the estimates in our case. We mention that a major difficulty will come
from the fact that ``small'' seeds are not excluded in the hierarchy growth.
\end{remark}

\begin{remark}
Our approach gives a much improved upper bound, compared to \cite
{Mountford93,Mountford95}, for the semi-oriented bootstrap
percolation. But at this point we do not have a sharp threshold result,
due to the lack of a corresponding argument for the lower bound.
\end{remark}
\section*{Acknowledgments}
The first author would like to thank Stanislav Smirnov for his constant support. The
second author is grateful to FOM for its longstanding support of the
Mark Kac seminar; this gave rise to a visit of the first author to the
Netherlands, which is when our collaboration started.

% imsref loaded by akundreckaite, 2012-03-09 12:42:50

%

%suskaldyti doi

\printaddresses

\end{document}

\begin{frontmatter}

\title{Erratum to
``Sharp metastability threshold for an~anisotropic
bootstrap percolation model''}
%\runtitle{}
\pdftitle{Erratum to ``Sharp metastability threshold for an~anisotropic
bootstrap percolation model''}

\begin{aug}
% Corresponding author: Aernout van Enter - avanenter@gmail.com% Updated by VTEXPTS2LaTeX.exe, 11.08.2015 12:26
%Updated by VTEXPTS2LaTeX.exe, 10.08.2015 13:24
\author[A]{\fnms{Hugo}~\snm{Duminil-Copin}\ead[label=e1]{hugo.duminil@unige.ch}}%,
%\author[]{\fnms{}~\snm{}\ead[label=]{}}
\and
\author[B]{\fnms{Aernout}~\snm{van Enter}\corref{}\ead[label=e2]{avanenter@gmail.com}}
%\runauthor{}
\affiliation{Universit\'e de Gen\`eve and University of Groningen}
%\dedicated{}
\address[A]{D\'epartement de Mathematiques\\
Universit\'e de Gen\`eve\\
2--4 Rue du L\`evre\\
Gen\`eve \\
Switzerland\\
\printead{e1}}
\address[B]{Johann Bernoulli instituut\\
\quad voor Wiskunde en Informatica\\
University of Groningen\\
P.O. Box 407\\
9700AK Groningen\\
Netherlands \\
\printead{e2}}
\end{aug}

% HISTORY:
%
\received{\smonth{6} \syear{2015}}% Updated by VTEXPTS2LaTeX.exe,
%10.08.2015 13:24
%\revised{\smonth{} \syear{}}

% ABSTRACT
\begin{abstract}
We provide an Erratum, correcting how our main result generalises and
correct some steps in the proof.
\end{abstract}

% KEYWORDS
\begin{keyword}[class=AMS]
%\kwd[Primary ]{}
\kwd{60K35}
\kwd{83B43}
\kwd{83C43}
%\kwd[; secondary ]{}
\end{keyword}
\begin{keyword}
\kwd{Bootstrap percolation}
\kwd{anisotropy}
\kwd{sharp threshold}
\kwd{metastability}
\end{keyword}
\end{frontmatter}

Although Theorem~1.1 as stated in our paper \cite{DE} is essentially
correct, the generalisation announced at the end of page 1219 and the
beginning of page 1220 to the situation where the site $(m,n)$ gets
occupied if $k+1$ sites among the $2k+2$ sites $(m+1,n)$, $(m,n-1)$ and
$(m-k,n),\ldots, (m-1,n), (m+1,n),\ldots,(m+k,n)$ are occupied,
is incorrect as stated\vspace*{1.5pt} for $k \gneq2$.

It should read that (in probability)
$\frac{\log T}{({1}/{p})(\log({1}/{p}))^2 }$ (which is the
correct expression, and should also read like this in the theorem)
tends to $\frac{(k-1)^2}{4(k+1)}$. In other words,
the constant
$\frac{1}{4(k+1)}$ given in our paper on the right-hand side of the
unnumbered equation on page 1220 should be replaced by $\frac
{(k-1)^2}{4(k+1)}$.
We thank Rob Morris for pointing this out to us.

Also, the definition of weakly connected in page 1226 should be
modified as follows. Two occupied points $x,y\in\mathbb{Z}^2$ are
weakly connected if $x\in y+\mathcal{N}$ or
there exists $z\in\mathbb{Z}^2$ such that $x,y\in z+\mathcal{N}$.
Moreover, in the definition of $A_k(R_1, R_2)$, we need $R_2$ to be
$k$-crossed, rather than $k$-vertically crossed as was written.

We thank Tim Hulshof for pointing this out to us.

% imsref loaded by daiva.urboniene, 2015-08-10 13:34:08

%\begin{appendix}
%\section{}
%\end{appendix}

%\section*{Acknowledgments}

%\begin{supplement}[id=suppA]
%\slink[doi]{10.1214/00-AOPXXXXSUPP} %[doi,text={...}] - jei reikia
%suskaldyti doi
%\sfilename{aopXXXX\_supp.pdf}
%\end{supplement}

%\begin{thebibliography}{99}
%\bibitem[\protect\citeauthoryear{}{}]{r1}
%\bibitem{r1}
%\end{thebibliography}

\printaddresses